\documentclass[11pt]{amsart}

\usepackage{amsthm}
\usepackage{amsmath}
\usepackage{amssymb}
\usepackage{color}

\newtheorem{thm}{Theorem}[section]
\newtheorem{theorem}[thm]{Theorem}

\newtheorem{corollary}[thm]{Corollary}

\newtheorem{lemma}[thm]{Lemma}
\newtheorem{proposition}[thm]{Proposition}

\theoremstyle{remark}
\newtheorem{remark}[thm]{Remark}

\newcommand{\norm}[1]{\|#1\|}

\newcommand{\RR}{\mathbb R}

\newcommand{\e}{\varepsilon}

\newcommand{\ip}[2]{\left\langle#1,#2\right\rangle}
\newcommand{\abs}[1]{\left| #1 \right|}

\begin{document}

\title{Controled scaling of Hilbert space frames for $\RR^2$}
\author[Casazza, Xu
 ]{Peter G. Casazza and Shang Xu}
\address{Department of Mathematics, University
of Missouri, Columbia, MO 65211-4100}

\thanks{The authors were supported by
 NSF DMS 1609760 and 1906725}

\email{Casazzap@missouri.edu, sxz59@mail.missouri.edu}


\begin{abstract}
A Hilbert space frame on $\RR^n$ is {\it scalable} if we can scale the vectors to make them a tight frame.  There are known
classifications of scalable frames.  There are two basic questions here which have never been answered in any $\RR^n$:
\begin{enumerate}
\item Given a frame in $\RR^n$, how do we scale the vectors to minimize the condition number of the frame?  I.e. How do we
scale the frame to make it as tight as possible?
\item If we are only allowed to use scaling numbers from the interval $[1-\epsilon,1+\epsilon]$, how do we scale the frame to
minimize the condition number?
\end{enumerate}
We will answer these two questions in $\RR^2$ to begin the process towards a solution in $\RR^n$.
\end{abstract}

\maketitle

\section{Introduction}

A family of vectors $\{\phi_i\}_{i=1}^m$ in an n-dimensional Hilbert space $H^n $  is a {\it frame} if there are constants
$0<A\le B<\infty$ satisfying:
\[ A \|\phi\|^2 \le \sum_{i=1}^m|\langle \phi,\phi_i\rangle|^2 \le B \|\phi\|^2, \mbox{ for all }\phi \in H^n.\]
If $A=B$ this is an {\it A-tight frame} and if $A=B=1$ this is a {\it Parseval frame}.  The largest A and smallest B satisfying
this inequalty are called the {\it lower} (respectively, {\it upper}) {\it frame bound}.  The {\it analysis operator} of the frame is the
operator $T:H^n\rightarrow \ell_2(m)$ given by $T\phi= (\langle \phi,
\langle \phi_i)_{i=1}^m$.  The {\it synthesis operator} of the frame
is $T^*$ and satisfies $T^*(a_i)=\sum_{i=1}^ma_i\phi_i$.  The {\it frame operator} is $S=TT^*$ and is the
positive, self-adjoint invertible operator $S:H^n\rightarrow H^n$ given by:
\[ S(\phi)=\sum_{i=1}^m\langle \phi, \phi_i\rangle \phi_i.\]
B turns out to be the largest eigenvalue of S and A is the smallest eigenvalue of S.
The quotient $B/A$ is called the {\it condition number}.  It is known \cite{CL} that a frame is A-tight if and only if 
$S\phi= A\phi$ for all $\phi \in H$.

Once at a meeting David Larson defined scaling for frames and asked:  {\it Which frames are scalable}?  A frame is
{\it scalable} if we can change the lengths of the frame vectors to form a tight frame.  Since then, much work has been
done on this problem \cite{CC,CCH,CK,CKO,CKL,DK,KOP,KO}.  The reason we like this is because the condition number
heavily determines the complexity of reconstruction.  But if the frame is tight, the condition number is 1 and the
frame operator is $S=AI$.  So reconstruction is trivial.
But there are problems with all of this.  First, the results in
this area invariably end up sending a hugh portion of the frame vectors to zero. 
But, this problem
grew out of applications where the frame is being used to analyze signals and in pratice we cannot just set a bunch of the
frame vectors equal to zero (or make them very small) and still be able to do signal processing.  Second, most of the results
in this are are not really solutions to the problem.  They are really just {\it equivalent formulations} of the problem.  That is,
they often read like:  A frame is scalable if and only if a certain operator does ...  But it is no easier to find and check this
operator than to try to figure out how to scale the vectors.  Third, in most cases the frame is not scalable.  And in this case
the real problem is to scale the frame so as to minimize the condition number since this gives the least complexity for
using the frame.  Fourth, in practice we can only scale the frame vectors a small amount.  So the real problem here is
to find the scaling which minimizes the condition number if the scaling is restricted to the interval
$[(1-\epsilon)\|\phi_i\|,(1+\epsilon)\|\phi_i\|]$.
We will address these shortcomings in the theory of scaling here.   A good reference for frame theory is \cite{C,CL}. 

\section{Some General Results}

We first consider a family of vectors which lies in an open quadrant.

\begin{proposition}
If after a change in signs $\{\phi_i\}_{i=1}^m$ lies in an open quadrant than this family is not scalable.
\end{proposition}

\begin{proof}
After a rotation we may assume the vectors lie in the first quadrant and no vector lies on $e_1=(1,0)$. Now, rotating $\phi=e_1$
a tiny amount towards $(0,1)$ so it does not pass any of our vectors we have:
\[ \sum_{i=1}^m|\langle \phi, \phi_i\rangle|^2 > \sum_{i=1}^m|\langle e_1,\phi_i\rangle|^2,\]
and so the frame is not tight.
\end{proof}

There is a simple way to see when a frame is tight.

\begin{theorem}\label{scalability_conditions}
 A frame $\{\phi_i\}_{i=1}^m$ where $\phi_i=(a_i,b_i) \in \RR^2$ is tight if and only if the following two conditions hold:
\begin{enumerate}
\item $\sum_{i=1}^M a_i^2 = \sum_{i=1}^M b_i^2 = A$
 
\item $\sum_{i=1}^M a_ib_i = 0$
\end{enumerate}
\end{theorem}

\begin{proof}
Let $e_1=(1,0)$ and $e_2=(0,1)$.
Assume our frame is A-tight.  Then
\[ \sum_{i=1}^ma_i^2=\sum_{i=1}^m|\langle e_1,\phi_i\rangle|^2 = A\|e_1\|=A\|e_2\|^2=\sum_{i=1}^m|\langle e_2,\phi_i\rangle|^2
=\sum_{i=1}^mb_i^2.\] 
Property two is true of all tight frames \cite{CL}.

Conversely, given (1) and (2), for any $\phi=(a,b)\in H^2$ we have
\begin{align}
S\phi&= \sum_{i=1}^m\langle \phi,\phi_i\rangle \phi_i\\
&= \sum_{i=1}^m(aa_i+bb_i)(a_i,b_i)\\
&= \left ( \sum_{i=1}^maa_i^2+bb_ia_i, aa_ib_i+bb_i^2\right )\\
&=(aA,bA)=A\phi.
\end{align}
So the frame is A-tight. 
\end{proof}

Earlier we mentioned that minimizing the condition produces the frame which is the closest to being tight.  Now we will make
this statement formal.

\begin{lemma}\label{min_norm_difference_over_scalars}
  Given a frame $\phi = \{\phi_i\}_{i=1}^m$ in $\RR^n$ with frame operator $S$
  and frame bounds $A$ and $B$
  \[\inf_{c} \left \| S - c Id\right \| = \frac{B-A}{2}\] and the minimum is
  attained when $c = \frac{B+A}{2}$. Thus, if $\frac{B+A}{2} = 1$ then $ \| S -Id \| = \frac{B-A}{2}$
\end{lemma}
\begin{proof}
  Let $\{e_i\}_{i=1}^{n}$ be the eigenvectors of $S$ with corresponding
  eigenvalues $\{\lambda_i\}_{i=1}^{n}$. Then observe that for $S - c Id$ the
  eigenvectors are also $\{e_i\}_{i=1}^{n}$, but with eigenvalues $\{\lambda_i -
  c\}_{i=1}^{n}$. Then we may assume that the eigenvalues are ordered such that
  $B = \lambda_1 \geq \lambda_2 \geq ... \geq \lambda_n = A.$ So, $$\norm{S - c
    Id} = \max_{1 \leq i \leq n} \abs{\lambda_i - c} =
  \max\{\abs{A-c},\abs{B-c}\}.$$ If $c \notin [A,B],$ then moving towards
  $[A,B]$ will increase the max. If $$B \geq c > \frac{A+B}{2}$$
  then $$\abs{A-c} > \frac{A+B}{2} - A = \frac{B-A}{2}.$$ If $$\frac{A+B}{2} > c
  \geq A$$ then $$ \frac{B-A}{2} = \frac{A+B}{2} - A > \abs{c - A}.$$ That is
  $c = \frac{A+B}{2}$ and $\norm{S - c Id} = \frac{B-A}{2}.$ Letting $c=1$
  gives the desired result.
\end{proof}
\begin{proposition}\label{eq_of_op_norm}
  Given a frame $\Phi = \{\phi_i\}_{i=1}^m$ in $\RR^n$ let $C$ be the set of
  sequences of scalars $\alpha = \{a_i\}_{i=1}^m$ with $a_i>0$ for all $i \in
  \{1,..,m\}$. Let $A_{\alpha}$ and $B_{\alpha}$ be the frame bounds of
  $\alpha\phi = \{a_i \phi_i\}_{i=1}^m$ and $S_{\alpha}$ be $\alpha\Phi$'s frame
  operator. Then if $\alpha, \beta \in C$ the following are equivalent:
  \begin{enumerate}
  \item
    \[ \frac{B_{\alpha}}{A_{\alpha}} \leq \frac{B_{\beta}}{A_{\beta}}\]
  \item
    \[ \left \|S_{\alpha} - Id \right\| \leq \left \|S_{\beta} - Id \right\| \]
  \end{enumerate}
  In other words, if a scaling minimizes the condition number then it also
  minimizes the frame operator's distance from the identity.
\end{proposition}

\begin{proof}
  First note that if $\alpha \in C$ and $\frac{A_{\alpha} + B_{\alpha}}{2} = c
  \neq 1$ then we can multiply $\alpha$ by $\frac{1}{\sqrt{c}}$ to get a new
  scaling $\alpha'$ such that $\frac{A_{\alpha'} + B_{\alpha'}}{2} = 1$ and
  $\frac{B_{\alpha}}{A_{\alpha}} = \frac{B_{\alpha'}}{A_{\alpha'}}$. To see this
  let $e$ be an eigenvector of $S_{\alpha}$ with eigenvalue $\lambda$. Then
  $S_{\alpha'}e = \sum_{i=1}^m \ip{e}{\frac{a_i}{\sqrt{c}}\phi_i}
  \frac{a_i}{\sqrt{c}}\phi_i = \frac{1}{c}\sum_{i=1}^m \ip{e}{a_i\phi_i}
  a_i\phi_i = \frac{1}{c} S_{\alpha} e = \frac{\lambda}{c} e$. Since
  $A_{\alpha}$, $B_{\alpha}$, $A_{\alpha'}$, and $B_{\alpha'}$ are eigenvalues
  of their respective operators the statement follows.
    
  Thus we can focus our attention on scalings $\alpha$ and $\beta$ with frame
  bounds satisfying $\frac{A_{\beta} + B_{\beta}}{2} = \frac{A_{\alpha} +
    B_{\alpha}}{2} = 1$. In this case $\|S_{\gamma} - Id \| = \frac{B_{\gamma} -
    A_{\gamma}}{2}$ by lemma \ref{min_norm_difference_over_scalars}. Further
  more we have $\frac{B_{\gamma} + A_{\gamma}}{2} = 1$ so $B_{\gamma} = 2 -
  A_{\gamma}$. Next,
  \[\|S_{\gamma} - Id \| = \frac{B_{\gamma} - A_{\gamma}}{2} = \frac{2 -
      A_{\gamma} - A_{\gamma}}{2} = 1 - A_{\gamma}\]
  and \[\frac{B_{\gamma}}{A_{\gamma}} = \frac{2- A_{\gamma}}{A_{\gamma}} =
    \frac{2}{A_{\gamma}} - 1 .\] Thus:
  
    \[\begin{gathered}
        \frac{B_{\alpha}}{A_{\alpha}} \leq \frac{B_{\beta}}{A_{\beta}}\ \iff
        \frac{2}{A_{\alpha}} - 1 \leq \frac{2}{A_{\beta}} - 1 \iff A_{\beta}
        \leq
        A_{\alpha}  \\
        \iff 1 - A_{\alpha} \leq 1 - A_{\beta} \iff \|S_{\alpha} - Id \| \leq
        \|S_{\beta} - Id \|
      \end{gathered}\]
  
  \end{proof}

Finally,

\begin{theorem}
	Let $\{ \phi_{i} \}_{i=1}^{M}$ be a frame in $\mathbb{R}^{n}$ and denote by $S$ its frame operator.  Let  $\lambda_1\geq \cdots \geq \lambda_n$ be the eigenvalues of $S$. Then we have
	$$ \min_{c\geq 0}\| I_n-cS\|^2 = n - \frac{\left( \sum_{i=1}^{n} \lambda_{i} \right)^{2}}{\sum_{i=1}^{n} \lambda_{i}^{2}} .$$
\end{theorem}
\begin{proof}

$$ \min_{c \geq 0} \| I_n-cS\|^2 = \sum_{i=1}^{n} | 1-c \lambda_{i} |^{2} = n + c^{2} \sum_{i=1}^{n} \lambda_{i}^{2} - 2c \sum_{i=1}^{n} \lambda_{i}$$

Hence, setting

$$ \frac{d}{dc} \min_{c \geq 0} \| I_n-cS\|^2 = 2c \sum_{i=1}^{n} \lambda_{i}^{2} -2\sum_{i=1}^{n} \lambda_{i} = 0  $$
we get 

\[c=\frac{\left(\sum_{i=1}^n\lambda_i\right)^2}{\sum_{i=1}^n\lambda_i^2},\]

which gives the desired result.
\end{proof}
% This part is new vvv
\subsection{Motivation for the restricted scaling factors}
Suppose $\{ \Phi_i\}_{i=1}^M$ is an unscalable frame in $\RR^2$, where the outer two vectors are of the form $(1,0), (\e, \sqrt{1-\e^2})$, and the rest of the vectors all lie between. Suppose one can scale each vector
by any number $\alpha_i$, then one can either set the middle vectors to zero, which from the following results, will produce the lowest condition number. Another alternative result is to set $\alpha$ for the outer two vectors very large, which is equivalent to scaling the middle vectors to zero. However, doing so is highly impractical, and thus an upper and lower bound of the scaling factor must be implemented in order to deliver a result that is useful. 
%This part is new^^^

\section{Two Vectors}

\noindent {\bf Assumptions for this section}:
Let $S$ be the frame operator of a frame in $\mathbb{R}^2$ consisting of $\phi_1 = (k,0)^T$, $\phi_2 = (a,b)^T$ where $\|\phi_1\|\leq \|\phi_2\| = 1$.  Hence, it follows that $1 = \|\phi_2\|^2 = a^2 + b^2$, that is: $b^2 = 1 - a^2$. It is also assumed that $0 < k < 1$.

\begin{proposition}
The eigenvectors of $S$ are a linear combination of $\phi_1$ and $\phi_2$, namely,
$$
\frac{v}{\|v\|} = w \phi_1 + \phi_2 = w \binom{k}{0} + \binom{a}{b} = \binom{wk + a}{b} = (wk + a, b)^T.
$$
The eigenvalues are $1 + wka, k^2 -wka$.
\end{proposition}

\begin{proof}
First of all let us notice that the frame operator matrix has the following form, where $T$ and $T^\ast$ are the analysis and synthesis operators correspondingly:
$$
S = T^\ast \times T =
$$

\[
=\begin{bmatrix}
k & a\\
0 & b
\end{bmatrix}\times
\begin{bmatrix}
k & 0\\
a & b
\end{bmatrix} =
\begin{bmatrix}
k^2 + a^2 & ab\\
ab & b^2
\end{bmatrix}.
\]

Since by Proposition 4.3 \cite{CL}, it follows that $\phi_1$ and $\phi_2$ span $\mathbb{R}^2$, the eigenvectors can always be written as a linear combination of the two basis.

Let us find the first eigenvalue corresponding to $v = (wk+a, b)^T$. 

\[
Sv = \lambda_1 v \Longleftrightarrow \begin{bmatrix}
k^2 + a^2 -\lambda_1 & ab\\
ab & b^2 - \lambda_1
\end{bmatrix}
\times
\begin{bmatrix}
wk+a\\
b
\end{bmatrix}
= 0 \Longleftrightarrow
\]

\[
\begin{bmatrix}
(k^2 + a^2 -\lambda_1)(wk + a) + ab^2\\
ab(wk + a) + b(b^2 - \lambda_1)
\end{bmatrix}
= 0 \Longleftrightarrow
\]

\begin{equation*}
 \begin{cases}
   (k^2 + a^2 - \lambda_1)(wk+a)+ ab^2 = 0, 
   \\
   ab(wk + a) + b(b^2 - \lambda_1) = 0.
 \end{cases}
\end{equation*}

Out of the second equation it follows that:
$$
\lambda_1 = awk + a^2 + 1 - a^2 = 1 +awk
$$

Out of the first equation it follows that
$$
\lambda_1(wk + a) = k^3w + ak^2 + a^2wk + a,
$$
that is:
$$
(1 + wka)(wk + a) = k^3w + ak^2 + a^2wk + a
$$
$$
wk + a + w^2 k^2 a +wka^2 = k^3w + ak^2 + a^2wk + a
$$
$$
wk + w^2 k^2 a = k^3w + ak^2
$$
$$
w + w^2 k a = k^2w + ak
$$
$$
ka = (w^2 k)a + w(1 - k^2) \eqno{(1.1)}
$$

Let us find the second eigenvalue corresponding to $\hat{v} = (-b, wk+a)^T$.

\[
Sv = \lambda_2 v \Longleftrightarrow \begin{bmatrix}
k^2 + a^2 -\lambda_2 & ab\\
ab & b^2 - \lambda_2
\end{bmatrix}
\times
\begin{bmatrix}
-b\\
wk + a
\end{bmatrix}
= 0 \Longleftrightarrow
\]

\[
\begin{bmatrix}
-(k^2 + a^2 -\lambda_2)b + ab(wk + a)\\
-ab^2 + (wk + a)(b^2 - \lambda_2)
\end{bmatrix}
= 0 \Longleftrightarrow
\]

\begin{equation*}
 \begin{cases}
   -(k^2 + a^2 -\lambda_2)b + ab(wk + a) = 0, 
   \\
   -ab^2 + (wk + a)(b^2 - \lambda_2) = 0.
 \end{cases}
\end{equation*}

Out of the first equation, it follows that:
$$
-k^2 - a^2 +\lambda_2 + awk + a^2 = 0
$$

$$
\lambda_2 = k^2 -awk
$$

Out of the second equation, it follows that 
$$
-ab^2 + (wk + a)b^2 = \lambda_2 (wk + a)
$$
$$
wkb^2 = \lambda_2 (wk + a)
$$
$$
\lambda_2 = \frac{wk(1  - a^2)}{(wk + a)}
$$
$$
k^2 - awk = \frac{wk(1  - a^2)}{(wk + a)}
$$
$$
k - aw = \frac{w(1  - a^2)}{(wk + a)}
$$
$$
(k - aw)(wk + a) = w  - w a^2
$$
$$
k^2w + ak - a w^2 k -a^2 w = w  - w a^2
$$
$$
k^2w + ak - a w^2 k = w
$$
$$
(ak)w^2 + (1 - k^2)w - ka = 0 \eqno{(1.2)}
$$

Note that $w^2 = 1$ when $k=1$.

\vspace{.3 in}

Solving $(1.1)$ or $(1.2)$ for $w$ gives:
$$
(ka)w^2 + (1 - k^2)w - ka = 0
$$
$$
D = (1 - k^2)^2 + 4 k^2 a^2
$$
$$
w_{1,2} = \frac{k^2 -1 \pm \sqrt{D}}{2ka} 
$$
$$
w := w_1 = \frac{k^2 -1 \pm \sqrt{D}}{2ka} 
$$
$$
wka = \frac{k^2 - 1 + \sqrt{D}}{2}
$$
$$
\lambda_1 = 1 + wka = 1 + \frac{k^2 -1 + \sqrt{D}}{2} = \frac{k^2 + 1 + \sqrt{D}}{2}
$$
$$
\lambda_2 = k^2 - wka = k^2 - \frac{k^2 - 1 + \sqrt{D}}{2} = \frac{k^2 + 1 - \sqrt{D}}{2}
$$
\end{proof}
It can be noticed that $\lambda_1 > \lambda_2$, and hence the \textbf{condition number} $\frac{B}{A} = \frac{\lambda_1}{\lambda_2}$.
$$
\frac{B}{A} = \frac{\lambda_1}{\lambda_2} = \frac{\frac{k^2 + 1 + \sqrt{D}}{2}}{\frac{k^2 + 1 - \sqrt{D}}{2}}: = f(k, a)
$$
$$
f(k, a) = \frac{\frac{k^2 + 1 + \sqrt{(1-k^2)^2 + 4k^2a^2}}{2}}{\frac{k^2 + 1 - \sqrt{(1-k^2)^2 + 4k^2a^2}}{2}} = \frac{k^2 + 1 + \sqrt{(1-k^2)^2 + 4k^2a^2}: = Num}{k^2 + 1 - \sqrt{(1-k^2)^2 + 4k^2a^2}: = Den}
$$

\vspace{.3 in}
\begin{proposition}
With the current assumption, let all but $a$ be fixed, then $\langle \phi_1, \phi_2 \rangle$ increases with the condition number.
\end{proposition}
\begin{proof}
First, since we are writing $\phi_1 = (k,0)$, and we are assuming both vectors lie in a quadrant, the inner product is simply $\langle \phi_1, \phi_2 \rangle = ak$. Now, let us look at the derivative of the condition number, $f(k,a)$ with respect to $a$. Note that by our assumption, $a \in (0,1)$
$$
f_a'(k,a) = \frac{1}{\Big(k^2 + 1 - \sqrt{(1-k^2)^2 + 4k^2a^2}\Big)^2} \times 
$$
$$
\times\Big(\frac{8k^2 a}{2\sqrt{(1-k^2)^2 + 4k^2 a^2}}\Big(k^2 + 1 - \sqrt{(1-k^2)^2 + 4k^2a^2}\Big) +
$$
$$
+ \Big(k^2 + 1 + \sqrt{(1-k^2)^2 + 4k^2a^2}\Big)\frac{8k^2 a}{2\sqrt{(1-k^2)^2 + 4k^2 a^2}}\Big) =
$$
$$
= \frac{1}{Den^2}\Big(\frac{8k^2a}{2\sqrt{D}}\cdot Den + \frac{8k^2a}{2\sqrt{D}}\cdot Num\Big) =
$$
$$
= \frac{1}{Den^2}\times\frac{8k^2a}{2\sqrt{D}}\Big(Den + Num\Big) =
$$
$$
= \frac{1}{Den^2}\times\frac{8k^2a}{2\sqrt{D}}\Big(2(k^2 + 1)\Big) > 0,~if~ a > 0.
$$
Since we assumed the two vectors are in the same quadrant, $a$ is positive. The inner product also increase as $a \rightarrow 1$, which follows immediate from the basic properties of inner products. Hence, we have shown that the condition number is affected by the inner product. In $\RR^2$, this can also be interpreted as the angle between the two vectors.
\end{proof}

\vspace{.3 in}
\begin{proposition}
With the current assumptions, let all but the norm of $\phi_1$ be fixed, then the condition number decreases as the $\Vert\phi_1\Vert = k$ increases, given that $k \in (0,1]$.
\end{proposition}
\begin{proof}
Let us look at the derivative of $f(k,a)$ with respect to $k$.
$$
f_k'(k,a) = \frac{1}{\Big(k^2 + 1 - \sqrt{(1-k^2)^2 + 4k^2a^2}\Big)^2} \times 
$$
$$
\times\Big(\Big(2k + \frac{2(k^2 - 1) 2k + 8k a^2}{2\sqrt{(1-k^2)^2 + 4k^2 a^2}}\Big)\Big(k^2 + 1 - \sqrt{(1-k^2)^2 + 4k^2a^2}\Big) -
$$
$$
- \Big(k^2 + 1 + \sqrt{(1-k^2)^2 + 4k^2a^2}\Big)\Big(2k - \frac{2(k^2 - 1) 2k + 8k a^2}{2\sqrt{(1-k^2)^2 + 4k^2 a^2}}\Big)\Big) =
$$
$$
= \frac{1}{\Big(k^2 + 1 - \sqrt{(1-k^2)^2 + 4k^2a^2}\Big)^2}\times
$$
$$
\Big(\Big(2k(k^2 + 1) - 2k\sqrt{(1-k^2)^2 + 4k^2 a^2} + (k^2 + 1) \frac{2k(k^2 -1) + 4ka^2}{\sqrt{(1-k^2)^2 + 4k^2 a^2}} - 2k(k^2 -1) - 4ka^2\Big) -
$$
$$
- \Big(2k(k^2 + 1) + 2k\sqrt{(1-k^2)^2 + 4k^2 a^2} - (k^2 + 1) \frac{2k(k^2 -1) + 4ka^2}{\sqrt{(1-k^2)^2 + 4k^2 a^2}} - 2k(k^2 -1) - 4ka^2\Big)\Big) =
$$
$$
= \frac{1}{\Big(k^2 + 1 - \sqrt{(1-k^2)^2 + 4k^2a^2}\Big)^2}\times$$
 $$\Big(- 4k\sqrt{(1-k^2)^2 + 4k^2 a^2} + 2(k^2 + 1) \frac{2k(k^2 -1) + 4ka^2}{\sqrt{(1-k^2)^2 + 4k^2 a^2}}\Big) =
$$
$$
= \frac{4k}{\Big(k^2 + 1 - \sqrt{(1-k^2)^2 + 4k^2a^2}\Big)^2}\times
$$
$$
\times \frac{-(1-k^2)^2 - 4k^2 a^2 + (k^2 + 1)(k^2 -1) + 2a^2(k^2 + 1)}{\sqrt{(1-k^2)^2 + 4k^2 a^2}} =
$$
$$
= \frac{4k}{\Big(k^2 + 1 - \sqrt{(1-k^2)^2 + 4k^2a^2}\Big)^2}
\times \frac{-(1-k^2)^2 - 2k^2 a^2 + 2a^2 - (1 + k^2)(1 - k^2)}{\sqrt{(1-k^2)^2 + 4k^2 a^2}} =
$$
$$
= \frac{4k}{\Big(k^2 + 1 - \sqrt{(1-k^2)^2 + 4k^2a^2}\Big)^2}
\times \frac{-(1-k^2)^2 + 2a^2(1 - k^2) - (1 + k^2)(1 - k^2)}{\sqrt{(1-k^2)^2 + 4k^2 a^2}} =
$$
$$
= \frac{4k(1 - k^2)}{\Big(k^2 + 1 - \sqrt{(1-k^2)^2 + 4k^2a^2}\Big)^2}
\times \frac{-(1-k^2) + 2a^2 - (1 + k^2)}{\sqrt{(1-k^2)^2 + 4k^2 a^2}} =
$$
$$
= \frac{4k(1 - k^2)}{\Big(k^2 + 1 - \sqrt{(1-k^2)^2 + 4k^2a^2}\Big)^2}
\times \frac{2(a^2 - 1)}{\sqrt{(1-k^2)^2 + 4k^2 a^2}} < 0,
$$
since $1 = \|\phi_2\|^2 = a^2 + b^2$ implies that $a^2 = 1 - b^2 < 1$, and hence $|a| < 1$. This shows as $k$ increases, the condition number decreases.
\end{proof}

\begin{proposition}
The weight $w$, given to the shorter vector, $\phi_1$, is bounded from above by $\frac{k}{a}$, where $k$ is the norm of the shorter vector.
\end{proposition}

\begin{proof}
Since both eigenvalues must be positive, $\lambda_1 = 1 + awk > 0$ is trivial. For $\lambda_2$, we have $k^2 - wka > 0 \Longrightarrow \frac{k}{a} > w$.
\end{proof}

At this point, we know for any frame of two vectors in $\mathbb{R}^2$, all else constant, bringing the norms of the vectors together reduces the condition number.

\begin{theorem}
To minimize the condition number $\frac{B}{A}$, the vectors $\phi_1, \phi_2$ need to have the same length. Therefore, the condition number is $\frac{1 + |<\phi_1,\phi_2>|}{1 - |<\phi_1,\phi_2>|}$.
\end{theorem}

\begin{proof}
From above, since we are scaling the vectors, we are looking for values of $k$ to minimize the condition number. Since the condition number is a decreasing function of $k$, it follows we need to scale it to $\sup k = 1$, which by the previously established assumptions make the frame equal norm. Now that the frame is equal norm, $k = w =1$ ($w = 1$ follows either from $(1.1)$ or $(1.2)$), and the eigenvalues are $1 + a, 1 - a$. Since $\phi_1 = (1, 0)$, $<\phi_1, \phi_2> = a$. This also shows with equal norm frames, a smaller inner product between the two vectors will produce a smaller condition number.
\end{proof}

\begin{remark}\label{R1}
We have shown above that for a two vector frame as we move the length
of one vector towards the length of the other, the condition number
decreases.
\end{remark}

\begin{corollary}
Given $\phi_1, \phi_2$ with $a=\|\phi_1\|<\|\phi_2\|=b$, if we can only move them by $\varepsilon$, then the minimum condition number is:
\begin{enumerate}
\item If $b-a \le 2\epsilon$ then $(1+\delta)\phi_1,(1-\delta)\phi_2$ where $(1+\delta)a = (1-\delta)b$, gives the minimal
condition number.
\item If $b-a\ge 2 \epsilon$ then $(1+\epsilon)\phi_1,(1-\epsilon)\phi_2$ gives the minimal condition number.
\end{enumerate}
\end{corollary}

\section{Three Vectors}

We first see how to scale three vectors.  This argument is due to \cite{KO}.

\begin{proposition}
Given $\Phi=\{\phi_i\}_{i=1}^3$ in $\RR^2$ so that with any change of signs, $\{\pm \phi_i\}_{i=1}^3$ do not lie in a quadrant, then $\Phi$ is scalable.
\end{proposition}

\begin{proof}
By switchiong to $\pm \phi_i$, reindexing, and rotating if necessary, we may assume $e_1=(1,0)$ is an eigenvector and 
$\phi_1$ is in quadrant 1, $\phi_3$ is in quadrant 4, $\phi_2$ lies between these two, and the angle between $\phi_1,\phi_3$
is greater than $\pi$.  I.e. $\langle \phi_1,\phi_3\rangle <0$.  Now, $\langle \phi_1,\phi_2\rangle \langle \phi_2,\phi_3\rangle >0$.
I.e. Otherwise, $\{\phi_1,\phi_2,-\phi_3\}$ lie in a quadrant.  Choose $c>0$ with
\[ \langle \phi_1,\phi_2\rangle \langle \phi_2,\phi_3\rangle = -c\langle \phi_1,\phi_3\rangle.\]
\noindent {\bf Claim}:
\[ \left \{ \frac{c}{\sqrt{c^2\|\phi_1\|^2+|\langle \phi_1,\phi_2\rangle|^2}}\phi_1, \frac{1}{\sqrt{\|\phi_2\|^2 + c}}\phi_2,
\frac{c}{\sqrt{c^2\|\phi_3\|^2+|\langle \phi_2,\phi_3\rangle|^2}}\phi_3
\right \}\]
is a Parseval frame.
\vskip12pt
To see this, define $\{\psi_i\}_{i=1}^3$ in $\RR^2 \oplus \RR$ by:
\begin{enumerate}
\item \[ \psi_1= \phi_1\oplus (-\frac{1}{\sqrt{c}}\langle \phi_1,\phi_2\rangle\]
\item \[ \psi_2= \phi_2\oplus \sqrt{c}\]
\item \[ \psi_3= \phi_3\oplus (-\frac{1}{\sqrt{c}}\langle \phi_2,\phi_3\rangle\]
\end{enumerate}
Now
\[ \psi_1,\psi_2\rangle = \phi_1,\phi_2\rangle - \sqrt{c}\frac{1}{\sqrt{c}}\phi_1,\phi_2,\rangle =0.\]
\[ \langle \psi_2,\psi_3\rangle = \phi_2,\phi_3\rangle - \sqrt{c}\frac{1}{sqrt{c}}\phi_2,\phi_3\rangle =0.\]
\begin{align}
\langle \psi_1,\psi_3\rangle &= \phi_1,\phi_3\rangle -\frac{1}{\sqrt{c}}\langle \phi_1,\phi_2\rangle(-\frac{1}{\sqrt{c}})\langle
\phi_2,\phi_3\rangle\\
&= \langle \phi_1,l\phi_3\rangle + \frac{1}{c}\phi_1,\phi_2\rangle \langle \phi_2,\phi_3\rangle\\
&= \phi_1,\phi_3\rangle +\frac{1}{c}(-c\langle\phi_1,\phi_3\rangle=0
\end{align}
Since $\{\psi_i\}_{i=1}^3$ is orthogonal, normalizing it makes it an orthonormal basis of $\RR^3$ and so projecting these
vectors onto the first two coordinates is a Parseval frame - which is our set of vectors.
\end{proof}

\section{The General Case}

\begin{corollary}
If we have m vectors in $\RR^2$ and three (with any change in signs) do not lie in a quadrant, then we can scale these three to
be tight and the rest to be zero and we get a tight frame.
\end{corollary}
\begin{remark}
For any frame in $\RR^2$, the frame operator can be expressed as
\begin{align*}
    \begin{bmatrix}
    \sum_{i=1}^m \phi^2_{i,1} & \sum_{i=1}^m\phi_{i,1}\phi_{i,2}\\
     \sum_{i=1}^m\phi_{i,1}\phi_{i,2} & \sum_{i=1}^m \phi^2_{i,2}
    \end{bmatrix}
\end{align*}
Furthermore, the eigenvalues for any 2 by 2 matrix $A$ are the roots to the polynomial $x^2 - ~Trace~(A) + Det(A)$. 
\end{remark}

\begin{proposition}
For a frame in $\RR^2$ with $m$ vectors, all else held constant, the desired scaling for the vector $\phi_m = (\sqrt{x},0)$ is $$ x = \max\left( \sum_{i=1}^{m-1}\phi^2_{i,2} - \sum_{i=1}^{m-1}\phi^2_{i,1} + \frac{2(\sum_{i=1}^{m-1}\phi_{i,1}\phi_{i,2})^2}{\sum_{i=1}^{m-1} \phi^2_{i,2}},0 \right)$$
\end{proposition}

\begin{proof}
Let $a = a(x) = \sum_{i=1}^{m-1} \phi^2_{i,1}$, where $x$ is the square of the first component of the vector of interest, $b =\sum_{i=1}^m \phi^2_{i,2}$, $c =  (\sum_{i=1}^{m-1} \phi_{i,1}\phi_{i,2})^2$, we have
\begin{align*}
    \frac{B}{A} &= \frac{(a+b)+ \sqrt{(a+b)^2 - 4(ab - c)}}{(a+b)- \sqrt{(a+b)^2 - 4(ab - c)}}\\
    &= \frac{(a+b)+ \sqrt{a^2+2ab+b^2 - 4ab + 4c}}{(a+b)-\sqrt{a^2+2ab+b^2 - 4ab + 4c}}\\
    &= \frac{(a+b)+ \sqrt{(a-b)^2 + 4c}}{(a+b)-\sqrt{(a-b)^2 + 4c}}
\end{align*}
The last line above shows that the condition number is 1 will only happen if the column are orthogonal and if the square difference is zero, which follows from previous proven results.

Let $f(x) = \frac{(a(x)+b)+ \sqrt{(a(x)-b)^2 + 4c}}{(a(x)+b)-\sqrt{(a(x)-b)^2 + 4c}}$, where $a(x) = x + \sum \phi^2_{i,1}$ is a linear function of $x$ and $b,c$ are all constants independent of $x$, differentiating $f$ with respect to $x$ gives:
\begin{align*}
f'&= \frac{1}{(a+b)-\sqrt{(a-b)^2+4c)^2}}\times\\
&  [(1+\frac{a-b}{\sqrt{(a-b)^2+4c}})(a+b-\sqrt{(a-b)^2+4c})\\
&-(1-\frac{a-b}{\sqrt{(a-b)^2+4c}}+(a+b+\sqrt{(a-b)^2+4c})]
\end{align*}

\begin{align*}
    Top
    &= \left [a+b +  \frac{a^2-b^2}{\sqrt{(a-b)^2 + 4c}} - \sqrt{(a-b)^2 + 4c} - a+ b\right] \\
&- \left[ a+b -  \frac{a^2-b^2}{\sqrt{(a-b)^2 + 4c}} + \sqrt{(a-b)^2 + 4c} - a + b\right]\\
    &= \left[2b +  \frac{a^2-b^2}{\sqrt{(a-b)^2 + 4c}} - \sqrt{(a-b)^2 + 4c}  \right]\\
& - \left[ 2b -  \frac{a^2-b^2}{\sqrt{(a-b)^2 + 4c}} + \sqrt{(a-b)^2 + 4c}\right]\\
    &= 2\left (\frac{a^2-b^2}{\sqrt{(a-b)^2 + 4c}} - \sqrt{(a-b)^2 + 4c} \right)
\end{align*}
We are interested in when the Top part of $f'$ is zero, thus:
\begin{align*}
    0 & = \frac{a^2-b^2}{\sqrt{(a-b)^2 + 4c}} - \sqrt{(a-b)^2 + 4c}\\
    0&= a^2 - b^2 - ((a-b)^2 + 4c)\\
    0&= a^2 - b^2 - (a^2 -2ab + b^2 + 4c)\\
    0&= 2ab - 2b^2 - 4c\\
    a&= \frac{b^2 + 2c}{b}
\end{align*}
Since $a(x) = \sum_{i=1}^m \phi^2_{i,1}$, it is a linear function of $x$ with $\sum_{i=1}^{m-1} \phi^2_{i,1}$, back substituting the other constants, we have:
\begin{align*}
    x + \sum_{i=1}^{m-1} \phi^2_{i,1} &= (\frac{\sum_{i=1}^{m-1} \phi^2_{i,2})^2 + 2(\sum_{i=1}^{m-1}\phi_{i,1}\phi_{i,2})^2}{\sum_{i=1}^{m-1} \phi^2_{i,2}}\\
    x &= \sum_{i=1}^{m-1}\phi^2_{i,2} - \sum_{i=1}^{m-1}\phi^2_{i,1} + \frac{2(\sum_{i=1}^{m-1}\phi_{i,1}\phi_{i,2})^2}{\sum_{i=1}^{m-1} \phi^2_{i,2}} 
\end{align*}
\end{proof}
\begin{proposition}
Let $\{\Phi_i\}_{i=1}^m$ be an unscalable frame with $\phi_m = (\sqrt{x},0)$ and the other vectors are equal norm with $\Vert \phi_i \Vert^2 = k^2$, suppose we can't scale under $\Vert \phi \Vert^2 =k$, then we want to scale $\phi_m$ to $(k,0)$, making the frame equal norm. %needs work still
\end{proposition}
\begin{proof}
If $x > k^2$, we will have an increasing derivative for the condition number, as $ab - b^2 - 2c > 0$ as $a$ increase, which is exactly what we don't want.
\end{proof}
\begin{proposition}
If we add vectors between the two outer vectors in an unscalable frame then the condition number increases.
\end{proposition}
\begin{proof}
Start with a frame with $m$ vectors, it doesn't matter how you rotate it, the frame operator stays the same for the most part, so go back to the abc equation and now if you differentiate with respect to $a$, you will see $\frac{df}{da} > 0$ which means if you keep rotating the frame and add another $(\sqrt{x},0)$ vector will only affect $a$ and hence making the condition number go up. %need more rigor
\end{proof}

If we add vectors outside the given vectors then the condition number might decrease.  So we are not drowning in details, we will
just outline how to construct such examples.  Start with $\phi_1=\phi_2=(\frac{1}{\sqrt{2}},\frac{1}{\sqrt{2}})$.  The upper frame bound here is 2 and the lower frame bound is 0 so the condition number is $B/A=\infty$.  Now add two vectors to this family:  
$\phi_3=(1,0),
\phi_4=(01)$.  Then the upper frame bound of these four vectors is $B'=3$ and the lower frame bound is 1 and the condition number is
$B'/A'=3$.  For our example, we need these sets to form frames and lie in a quadrant.  
If we rotate $\phi_1,\phi_4$ towards $\phi_3$ by a tiny amount, then $\{\phi_i\}_{i=1}^2$ is a frame whose condition
number is greater than
$(2-\epsilon)/\epsilon$ and $\{\phi_i\}_{i=1}^4$ lies in the open first quadrant and
the condition number is less than $(3-\epsilon)/(1-\epsilon)$ which is
significantly smaller. 

 \begin{remark}\label{R2}
What we have shown above is that if we have m vectors in the first
quadrant and add a vector interior to the two outer vectors, then as the length
of the vector increases the condition number of the vectors increases.
\end{remark}

\begin{corollary}
If $\{\phi_i\}_{i=	1}^m$ is a non-scalable frame in $\RR^2$, the minimal condition number is
\[ \frac{B}{A}= \frac{1+|\langle \phi_i, \phi_j\rangle|}{1-|\langle \phi_i, \phi_j\rangle|},\]
where
\[ |\langle \phi_i, \phi_j\rangle|= max_{n\not= k}|\langle \phi_n,\phi_k\rangle|.\]
\end{corollary}

\begin{theorem}
If we can only scale the frame vectors sitting in the first quadrant
of $\RR^2$ by:  $(1+\epsilon)\phi_i,(1-\epsilon)\phi_i$,  we get minimal condition number by scaling the two outside vectors until their lengths
are as close as possible to each and scale all 
other vectors by $1-\epsilon$.
\end{theorem}

\begin{proof}
This follows from remarks \ref{R1}, \ref{R2}.
\end{proof}

\section{Conclusion}

  Two vectors in $\RR^2$ are scalable
if and only if they are orthogonal.  Otherwise, we minimize the condition number by making the vectors equal norm.  Three
vectors are scalable if and only if with any changes in signs they do not lie in a quadrant.  In general, a frame is scalable if
and only if after any changes in signs, they do not live in a quadrant.  Otherwise, to minimize the condition number we should
pick $\phi_i,\phi_j$ satisfying
\[ |\langle \phi_i, \phi_j\rangle |=min_{k\not= n}|\langle \phi_k,\phi_n\rangle|,\]
and scale these vectors to be equal length and set all other frame
vectors to zero
Finally we answered the question of how to minimize the condition number if there is a restriction on how much we can
scale the vectors.

\end{document}